\documentclass{amsart}
\usepackage{amsfonts}
\usepackage{amsmath}
\usepackage{amssymb}
\usepackage{amsthm}

\newtheorem{theorem}{Theorem}[section]
\newtheorem{prop}{Proposition}[section]
\newtheorem{lemma}{Lemma}[section]
\newtheorem{rem}{Remark}[section]

\newtheorem{cor}{Corollary}[section]

\begin{document}
\author{Mark Pankov}
\title{Characterization of isometric embeddings of Grassmann graphs}
\subjclass[2000]{15A04, 51M35}
\keywords{Grassmann graph, Johnson graph, apartment, isometric embedding}
\address{Department of Mathematics and Computer Sciences, University of Warmia and Mazury,
S{\l}oneczna 54, 10-710 Olsztyn, Poland}
\email{pankov@matman.uwm.edu.pl}

\begin{abstract}
Let $V$  be an $n$-dimensional left vector space over a division ring $R$.
We write ${\mathcal G}_{k}(V)$ for the Grassmannian formed by $k$-dimensional subspaces of $V$
and denote by $\Gamma_{k}(V)$ the associated Grassmann graph.
Let also $V'$  be an $n'$-dimensional left vector space over a division ring $R'$.
Isometric embeddings of $\Gamma_{k}(V)$ in  $\Gamma_{k'}(V')$ are classified in \cite{Pankov2}.
A classification of $J(n,k)$-subsets in ${\mathcal G}_{k'}(V')$,
i.e. the images of isometric embeddings of the Johnson graph $J(n,k)$ in $\Gamma_{k'}(V')$,
is presented in \cite{Pankov1}.
We characterize isometric embeddings of $\Gamma_{k}(V)$ in  $\Gamma_{k'}(V')$ as mappings which transfer
apartments of ${\mathcal G}_{k}(V)$ to $J(n,k)$-subsets of ${\mathcal G}_{k'}(V')$.
This is a generalization of the earlier result concerning apartments preserving  mappings \cite[Theorem 3.10]{Pankov-book}.

\end{abstract}

\maketitle

\section{Introduction}
Let $V$ be an $n$-dimensional left vector space over a division ring $R$ and let ${\mathcal G}_{k}(V)$ be
the Grassmannian formed by $k$-dimensional subspaces of $V$.
The associated Grassmann graph will be denoted by $\Gamma_{k}(V)$.
By classical Chow's theorem \cite{Chow}, every automorphism of $\Gamma_{k}(V)$ with $1<k<n-1$ is induced
by a semilinear automorphism of $V$ or a semilinear isomorphism of $V$ to the dual vector space $V^{*}$ and
the second possibility can be realized only in the case when $n=2k$.
The statement fails for $k=1,n-1$. In this case, any two distinct vertices of $\Gamma_{k}(V)$ are adjacent
and any bijective transformation of ${\mathcal G}_{k}(V)$ is an automorphism of $\Gamma_{k}(V)$.

Results closely related to Chow's theorem can be found in \cite{BH,Havlicek,Havlicek2,Huang,Lim,Kreuzer2,PankovGD},
see also \cite[Section 3.2]{Pankov-book}.

One of resent generalizations of Chow's theorem is the classification of
isometric embeddings of $\Gamma_{k}(V)$ in $\Gamma_{k'}(V')$, where $V'$ is an $n'$-dimensional left vector space
over a division ring $R'$ \cite{Pankov2}.
The existence of such embeddings implies that
\begin{equation}\label{eq1-1}
\min\{k,n-k\}\le \min\{k',n'-k'\},
\end{equation}
i.e. the diameter of $\Gamma_{k}(V)$ is not greater than the diameter of $\Gamma_{k'}(V')$.
The case $k=1,n-1$ is trivial:
every isometric embedding of $\Gamma_{k}(V)$ in $\Gamma_{k'}(V')$
is a bijection  to a clique of $\Gamma_{k'}(V')$.
If $1<k<n-1$ then isometric embeddings of $\Gamma_{k}(V)$ in $\Gamma_{k'}(V')$
are defined by semilinear $(2k)$-embeddings,
i.e. semilinear injections which transfer any $2k$ linearly independent vectors to linearly independent vectors.

A result of similar nature is obtained in \cite{Pankov1}.
This is the classification of the ima\-ges of
isometric embeddings of the Johnson graph $J(n,k)$ in the Grassmann graph $\Gamma_{k'}(V')$.
As above, we need \eqref{eq1-1} which guarantees that the diameter of $J(n,k)$
is not greater than the diameter of $\Gamma_{k'}(V')$.
The images of isometric embeddings of $J(n,k)$ in $\Gamma_{k'}(V')$ will be called
$J(n,k)$-{\it subsets} of ${\mathcal G}_{k'}(V')$.

Suppose that $1<k<n-1$ (the case $k=1,n-1$ is trivial).
If $n=2k$ then every $J(n,k)$-subset is an apartment in a parabolic subspace of ${\mathcal G}_{k'}(V')$
and we get an apartment of ${\mathcal G}_{k'}(V')$ if $n=n'$ and $k'=k,n-k$.
In the case when $n\ne 2k$, there are two distinct types of $J(n,k)$-subsets.

If $n=n'$ then every apartments preserving mapping of ${\mathcal G}_{k}(V)$ to ${\mathcal G}_{k}(V')$ with $1<k<n-1$
is induced by a semilinear embedding of $V$ in $V'$ or a semilinear embedding of $V$ in $V'^{*}$
and the second possibility can be realized only in the case when $n=2k$ \cite[Theorem 3.10]{Pankov-book}.
For $k=1,n-1$ this fails. By \cite{HK}, there are apartments preserving mappings of ${\mathcal G}_{1}(V)$
to itself which can not be defined by semilinear mappings.

Our main result (Theorem \ref{theorem-main}) characterizes isometric embeddings of $\Gamma_{k}(V)$ in $\Gamma_{k'}(V')$
as mappings which transfer apartments of ${\mathcal G}_{k}(V)$ to $J(n,k)$-subsets of ${\mathcal G}_{k'}(V')$.
As a consequence, we get  a generalization of the above mentioned result on apartments preserving mappings.

\section{Grassmann graph and Johnson graph}
\setcounter{equation}{0}
\subsection{Graph theory}
In this subsection we recall some concepts of the general graph theory.

A subset in the vertex set of a graph is called a {\it clique} if any two distinct vertices in this subset are adjacent
(connected by an edge). Every clique is contained in a maximal clique
(this is trivial if the vertex set is finite and we use Zorn lemma in the infinite case).

The {\it distance} between two vertices in a connected graph $\Gamma$ is defined as the smallest number $i$
such that there is a path consisting of $i$ edges and connecting these vertices.
The {\it diameter} of $\Gamma$ is the greatest distance between two vertices.

An {\it embedding} of a graph $\Gamma$ in a graph $\Gamma'$ is an injection of the vertex set of $\Gamma$ to the vertex set of $\Gamma'$
such that adjacent vertices go to adjacent vertices and non-adjacent vertices go to non-adjacent vertices.
Every surjective embedding is an isomorphism.
An embedding is said to be {\it isometric} if it preserves the distance between any two vertices.
Every embedding preserves the distances $1$ and $2$. Thus any embedding of a graph with diameter $2$ is isometric.

\subsection{Grassmann graph}
Let $V$ be an $n$-dimensional left vector space over a division ring $R$.
For every $k\in \{0,\dots,n\}$ we denote by ${\mathcal G}_{k}(V)$ the Grassmannian formed by $k$-dimensional subspaces of $V$.
Then ${\mathcal G}_{0}(V)=\{0\}$ and ${\mathcal G}_{n}(V)=\{V\}$.
In the case when $1\le k\le n-1$,
two elements of ${\mathcal G}_{k}(V)$ are  said to be {\it adjacent} if their intersection is $(k-1)$-dimensional
(this is equivalent to the fact that their sum is $(k+1)$-dimensional).

The {\it Grassmann graph} $\Gamma_{k}(V)$ is the graph whose vertex set is ${\mathcal G}_{k}(V)$
and whose edges are pairs of adjacent $k$-dimensional subspaces.
The graph  $\Gamma_{k}(V)$ is connected, the distance $d(S,U)$ between two vertices $S,U\in {\mathcal G}_{k}(V)$ is equal to
$$k-\dim(S\cap U)=\dim(S+U)-k$$
and the diameter of $\Gamma_{k}$ is equal to $\min\{k,n-k\}$.

Let $V^{*}$ be the dual vector space. This is an $n$-dimensional left vector space over the opposite division ring $R^{*}$
(the division rings $R$ and $R^{*}$ have the same set of elements and the same additive operation, the multiplicative operation $*$ on $R^{*}$
is defined by the formula $a*b:=ba$ for all $a,b\in R$).
The second dual space $V^{**}$ is canonically isomorphic to $V$.

For a subset $X\subset V$ the subspace
$$X^{0}:=\{\;x^{*}\in V^{*}\;:\;x^{*}(x)=0\;\;\;\forall\; x\in X\;\}$$
is called the {\it annihilator} of $X$.
The dimension of $X^{0}$ is equal to the codimension of $\langle X\rangle$.
The annihilator mapping  of the set of all subspaces of $V$ to the set of all subspaces of $V^{*}$
is bijective and  reverses  the inclusion relation, i.e.
$$S\subset U\;\Longleftrightarrow\;U^{0}\subset S^{0}$$
for any subspaces $S,U\subset V$.
Since $S^{00}=S$ for every subspace $S\subset V$,
the inverse bijection is also the annihilator mapping.
The restriction of the annihilator mapping to each ${\mathcal G}_{k}(V)$ is an isomorphism of $\Gamma_{k}(V)$ to $\Gamma_{n-k}(V^{*})$.

\begin{lemma}\label{lemma2-1}
If $S_1,\dots,S_m$ are subspaces of $V$ then
$$(S_{1}+\dots+S_{m})^{0}=(S_{1})^{0}\cap\dots\cap (S_{m})^{0},$$
$$(S_{1}\cap\dots\cap S_{m})^{0}=(S_{1})^{0}+\dots+(S_{m})^{0}.$$
\end{lemma}

Consider incident subspaces  $S\in {\mathcal G}_{s}(V)$ and $U\in {\mathcal G}_{u}(V)$ such that $s<k<u$.
We define
$$[S,U]_k:=\{\;P\in {\mathcal G}_{k}(V)\;:\;S\subset P\subset U\;\}.$$
In the case when $U=V$ or $S=0$, this subset will be denoted by $[S\rangle_{k}$ or $\langle U]_k$, respectively.
Subsets of such type are called {\it parabolic subspace} of ${\mathcal G}_{k}(V)$,
see \cite[Section 3.1]{Pankov-book}.

There is the  natural isometric embedding $\Phi^{U}_{S}$ of $\Gamma_{k-s}(U/S)$ in
$\Gamma_{k}(V)$ which sends every $(k-s)$-dimensional subspace of $U/S$ to the corresponding $k$-dimensional subspace of $V$.
In the case when $U=V$ or $S=0$, this embedding will be denoted by $\Phi_{S}$ or $\Phi^{U}$, respectively.
The image of $\Phi^{U}_{S}$ is the parabolic subspace $[S,U]_k$.

If $k=1,n-1$ then any two distinct vertices of $\Gamma_{k}(V)$ are adjacent.
In the case when $1<k<n-1$, there are precisely the following two types of maximal cliques
of $\Gamma_{k}(V)$:
\begin{enumerate}
\item[$\bullet$] the {\it star} $[S\rangle_{k}$, $S\in {\mathcal G}_{k-1}(V)$,
\item[$\bullet$] the {\it top} $\langle U]_{k}$, $U\in {\mathcal G}_{k+1}(V)$.
\end{enumerate}
The annihilator mapping transfers every parabolic subspace $[S,U]_k$ to the parabolic subspace $[U^{0},S^{0}]_{n-k}$;
in particular, it sends stars to tops and tops to stars.

\subsection{Johnson graph}
The {\it Johnson graph} $J(n,k)$ is the graph whose vertices are $k$-element subsets of
$\{1,\dots,n\}$ and whose edges are pairs of $k$-element subsets with $(k-1)$-element intersections.
The graph  $J(n,k)$ is connected, the distance $d(X,Y)$ between two vertices $X,Y$ is equal to
$$k-|X\cap Y|=|X\cup Y|-k$$
and the diameter of  $J(n,k)$ is equal to $\min\{k,n-k\}$.
The mapping
$$
X\to X^{c}:=\{1,\dots,n\}\setminus X
$$
is an isomorphism between $J(n,k)$ and $J(n,n-k)$.

If $k=1,n-1$ then any two distinct vertices of $J(n,k)$ are adjacent.
In the case when $1<k<n-1$,
there are precisely the following two types of maximal cliques of $J(n,k)$:
\begin{enumerate}
\item[$\bullet$] the {\it star} which consists of all vertices containing a certain $(k-1)$-element subset,
\item[$\bullet$] the {\it top}  which consists of all vertices contained in a certain $(k+1)$-element subset.
\end{enumerate}
The stars and tops of $J(n,k)$ consist of $n-k+1$ and $k+1$ vertices, respectively.
The isomorphism $X\to X^{c}$ transfers stars to tops and tops to stars.

Let $B$ be a base of $V$.
The associated {\it apartment} of ${\mathcal G}_{k}(V)$
consists of all $k$-dimensional subspaces spanned by subsets of $B$.
This is the image of an isometric embedding of $J(n,k)$ in $\Gamma_{k}(V)$.
We will use the following facts:
\begin{enumerate}
\item[$\bullet$] for any two $k$-dimensional subspaces of $V$ there is an apartment of ${\mathcal G}_{k}(V)$ containing both of them;
\item[$\bullet$] the annihilator mapping of ${\mathcal G}_{k}(V)$ to ${\mathcal G}_{n-k}(V^{*})$
transfers apartments to apartments.
\end{enumerate}

Let $[S,U]_k$ be a parabolic subspace of ${\mathcal G}_{k}(V)$.
Let also $B$ be a base of $V$ such that $S$ and $U$ are spanned by subsets of $B$.
The intersection of the corresponding apartment of ${\mathcal G}_{k}(V)$ with $[S,U]_k$ is said to be an {\it apartment} in the parabolic
subspace $[S,U]_k$.
This is the image of an isometric embedding of $J(u-s,k-s)$ in $\Gamma_{k}(V)$,
where $s=\dim S$ and $u=\dim U$. The mapping $\Phi^{U}_{S}$
establishes a one-to-one correspondence between apartments of ${\mathcal G}_{k-s}(U/S)$ and apartments
of the parabolic subspace $[S,U]_k$.

\subsection{Isometric embeddings of Johnson graphs in Grassmann graphs}
Let $V'$ be an $n'$-dimensional left vector space over a division ring $R'$.
Isometric embeddings of $J(n,k)$ in $\Gamma_{k'}(V')$ are classified in \cite{Pankov1}.
The existence of such embeddings implies that
the diameter of $J(n,k)$ is not greater than the diameter of  $\Gamma_{k'}(V')$, i.e.
\begin{equation}\label{eq2-1}
\min\{k,n-k\}\le\min\{k',n'-k'\}.
\end{equation}
Since $J(n,k)$ and $J(n,n-k)$ are isomorphic, we can suppose that $k\le n-k$. Then
$$k\le \min\{k',n-k,n'-k'\}.$$
The case $k=1$ is trivial: any two distinct vertices of $J(n,1)$ are adjacent and
every isometric embedding of $J(n,1)$ in $\Gamma_{k'}(V')$ is a bijection to a clique of $\Gamma_{k'}(V')$.

We say that a subset $X\subset V$ is $m$-{\it independent} if every $m$-element subset of $X$ is independent.
If $x_{1},\dots,x_{m}$ are linearly independent vectors of $V$ and
$$x_{m+1}=a_{1}x_{1}+\dots+a_{m}x_{m},$$
where each $a_{i}$ is non-zero, then $x_{1},\dots,x_{m+1}$ form an $m$-independent subset.
Every $n$-independent subset of $V$ consisting of $n$ vectors is a base of $V$.
By \cite[Proposition 1]{Pankov1}, if the division ring $R$ is infinite then for every natural integer $l\ge n$
there is an $n$-independent subset of $V$ consisting of $l$ vectors.

Suppose that $k< n-k$ and $X$ is a $(2k)$-independent subset of $V$ consisting of $l$ vectors.
Every $k$-element subset of $X$ spans a $k$-dimensional subspace and we denote by
${\mathcal J}_{k}(X)$ the set formed by all such subspaces.
This is the image of an isometric embedding of $J(l,k)$ in $\Gamma_{k}(V)$.
We will write ${\mathcal J}^{*}_{k}(X)$ for the subset of ${\mathcal G}_{n-k}(V^{*})$ consisting of the annihilators of elements from ${\mathcal J}_{k}(X)$.
If $X$ is a base of $V$
then ${\mathcal J}_{k}(X)$ and ${\mathcal J}^{*}_{k}(X)$ are apartments of ${\mathcal G}_{k}(V)$ and ${\mathcal G}_{n-k}(V^{*})$,
respectively.

The images of isometric embeddings of $J(n,k)$ in $\Gamma_{k'}(V')$
are called $J(n,k)$-{\it subsets} of ${\mathcal G}_{k'}(V')$.

\begin{theorem}[\cite{Pankov1}]\label{theorem2-1}
Let ${\mathcal J}$ be a $J(n,k)$-subset of ${\mathcal G}_{k'}(V')$ and $1<k\le n-k$.
In the case when $n=2k$, there exist $S\in {\mathcal G}_{k'-k}(V')$ and $U\in {\mathcal G}_{k'+k}(V')$
such that ${\mathcal J}$ is an apartment in the parabolic subspace $[S,U]_{k'}$, i.e.
$${\mathcal J}=\Phi^{U}_{S}({\mathcal A}),$$
where ${\mathcal A}$ is an apartment of ${\mathcal G}_{k}(U/S)$.
If $k<n-k$ then one of the following possibilities is realized:
\begin{enumerate}
\item[(1)] there exist $S\in {\mathcal G}_{k'-k}(V')$ and a $(2k)$-independent $n$-element subset $X\subset V'/S$
such that
$${\mathcal J}=\Phi_{S}({\mathcal J}_{k}(X));$$
\item[(2)] there exist  $U\in {\mathcal G}_{k'+k}(V')$ and
a $(2k)$-independent $n$-element subset $Y\subset U^{*}$  such that
$${\mathcal J}=\Phi^{U}({\mathcal J}^{*}_{k}(Y)).$$
\end{enumerate}
\end{theorem}
In the case when $1<k<n-k$, we say that ${\mathcal J}$ is a $J(n,k)$-subset of {\it first} or of {\it second type} if the corresponding
possibility is realized.
The annihilator mapping changes types of $J(n,k)$-subsets.

\begin{rem}\label{rem2-1}{\rm
Suppose that  $1<k<n-k$ and ${\mathcal J}\subset {\mathcal G}_{k'}(V')$ is a $J(n,k)$-subset of second type.
Let $U$ and $Y$ be  as in Theorem \ref{theorem2-1}.
The annihilators of vectors belonging to $Y$ form an $n$-element subset ${\mathcal Y}\subset{\mathcal G}_{k'+k-1}(U)$.
Every element of ${\mathcal J}$ can be presented as the intersection of $k$ distinct elements of ${\mathcal Y}$.
}\end{rem}

Let ${\mathcal C}$ be a maximal clique of $\Gamma_{k'}(V')$ (a star or a top).
As above, we suppose that ${\mathcal J}$ is a $J(n,k)$-subset of ${\mathcal G}_{k'}(V')$ and $1<k\le n-k$.
If ${\mathcal J}\cap {\mathcal C}$ contains more than one element than
it is a maximal clique of the restriction of $\Gamma_{k'}(V')$ to ${\mathcal J}$
(this restriction is isomorphic to $J(n,k)$). In this case,
we say that ${\mathcal J}\cap {\mathcal C}$ is a {\it star} or a {\it top} of ${\mathcal J}$
if ${\mathcal C}$ is a star or a top, respectively.

\begin{lemma}\label{lemma2-2}
Suppose that $1<k<n-k$.
If ${\mathcal J}$ is a $J(n,k)$-subset of first type then the stars and tops of ${\mathcal J}$ consist of $n-k+1$ and $k+1$ vertices, respectively.
In the case when ${\mathcal J}$ is a $J(n,k)$-subset of second type, the stars and tops of ${\mathcal J}$ consist of $k+1$ and $n-k+1$ vertices, respectively.
\end{lemma}

\begin{proof}
Easy verification.
\end{proof}

Lemma \ref{lemma2-2} shows that the two above determined classes of $J(n,k)$-subsets are disjoint.

\subsection{Isometric embeddings of Grassmann graphs}
Isometric embeddings of $\Gamma_{k}(V)$ in $\Gamma_{k'}(V')$ are classified in \cite{Pankov2}.
As in the previous subsection, we have \eqref{eq2-1} which implies that the diameter of $\Gamma_{k}(V)$ is not greater than the diameter of $\Gamma_{k'}(V')$.

A mapping $l:V\to V'$ is called {\it semilinear} if
$$l(x+y)=l(x)+l(y)$$
for all $x,y\in V$ and there is a homomorphism $\sigma:R\to R'$ such that
$$l(ax)=\sigma(a)l(x)$$
for all $a\in R$ and $x\in V$.
If $l$ is non-zero then  there is only one homomorphism $\sigma$ satisfying this condition.
Every non-zero homomorphism of $R$ to $R'$ is injective.

A semilinear injection of $V$ to $V'$ is said to be a {\it semilinear $m$-embedding} if
it transfers any $m$ linearly independent vectors to linearly independent vectors.
The existence of such mappings implies that $m\le n'$.
A semilinear $n$-embedding of $V$ in $V'$ will be called a {\it semilinear embedding}.
It maps every independent subset to an independent subset
which means that $n\le n'$.
For any natural integers $p\ge 3$ and $q$
there is a semilinear $p$-embed\-ding of a $(p+q)$-dimensional vector space
which is not a $(p+1)$-embedding \cite{Kreuzer1}.

Let $l:V\to V'$ be a semilinear $m$-embedding.
If $P$ is a $k$-dimensional subspace of $V$ and $k\le m$ then $\langle l(P)\rangle$
is a $k$-dimensional subspace of $V'$.
So, for every $k\in \{1,\dots,m\}$ we have the mapping
$$(l)_{k}:{\mathcal G}_{k}(V)\to {\mathcal G}_{k}(V')$$
$$P\to \langle l(P)\rangle$$
and the mapping
$$(l)^{*}_{k}:{\mathcal G}_{k}(V)\to {\mathcal G}_{n'-k}(V'^{*})$$
$$P\to \langle l(P)\rangle^{0}.$$
In the case when $k<m$, these mappings are injective and transfer adjacent subspaces to adjacent subspaces.
If $2k\le m$ then $(l)_{k}$ and $(l)^{*}_{k}$ are isometric embeddings of $\Gamma_{k}(V)$
in $\Gamma_{k}(V')$ and $\Gamma_{n'-k}(V'^{*})$, respectively.

\begin{theorem}[\cite{Pankov2}]\label{theorem2-2}
Let $f$ be an isometric embedding of $\Gamma_{k}(V)$ in $\Gamma_{k'}(V')$ and $1<k\le n-k$.
Then one of the following possibilities is realized:
\begin{enumerate}
\item[{\rm (1)}]
there exist $S\in {\mathcal G}_{k'-k}(V')$ and a semilinear $(2k)$-embedding $l:V\to V'/S$
such that $f=\Phi_{S}\circ(l)_{k}$;
\item[{\rm (2)}]
there exist $U\in {\mathcal G}_{k'+k}(V')$ and a semilinear  $(2k)$-embedding $s:V\to U^{*}$ such that
$f=\Phi^{U}\circ(s)^{*}_{k}$.
\end{enumerate}
In particular, if $n=2k$ then there exist incident $S\in {\mathcal G}_{k'-k}(V')$ and $U\in {\mathcal G}_{k'+k}(V')$
such that $f$ is induced by a semilinear embedding $l: V\to U/S$ or a semilinear embedding $s:V\to(U/S)^{*}$, i.e.
$$f=\Phi^{U}_{S}\circ (l)_k\;\mbox{ or }\;f=\Phi^{U}_{S}\circ (s)^{*}_k.$$
\end{theorem}

The case $k=1,n-1$ is trivial. The case when $1<k\le n-k$ is considered in Theorem \ref{theorem2-2}.
Suppose that $n-k<k<n-1$.
Since $\Gamma_{k}(V)$ and $\Gamma_{n-k}(V^{*})$ are canonically isomorphic,
every isometric embedding of $\Gamma_{k}(V)$ in $\Gamma_{k'}(V')$
can be considered as an isometric embedding of $\Gamma_{n-k}(V^{*})$ in $\Gamma_{k'}(V')$.
The latter embedding is one of the mappings described in Theorem \ref{theorem2-2}.
In contrast to the case when $1<k\le n-k$, we can not show that isometric embeddings of $\Gamma_{k}(V)$ in $\Gamma_{k'}(V')$
are defined by semilinear mappings of $V$.

\section{Main result}
Let $f$ be a mapping of ${\mathcal G}_{k}(V)$ to ${\mathcal G}_{k'}(V')$.
If the restriction of $f$ to every apartment of ${\mathcal G}_{k}(V)$ is an isometric embedding of $J(n,k)$ in $\Gamma_{k'}(V')$
then $f$ is an isometric embedding of $\Gamma_{k}(V)$ in $\Gamma_{k'}(V')$.
This follows from the fact that for any two elements of ${\mathcal G}_{k}(V)$ there is an apartment containing both of them.

We say that $f$ is a $J$-{\it mapping}
if it sends every apartment of ${\mathcal G}_{k}(V)$ to
a $J(n,k)$-subset.
Every isometric embedding of $\Gamma_{k}(V)$ in $\Gamma_{k'}(V')$ satisfies this condition.
Our main result states that this property characterizes isometric embeddings of $\Gamma_{k}(V)$ in $\Gamma_{k'}(V')$.

\begin{theorem}\label{theorem-main}
Every $J$-mapping of ${\mathcal G}_{k}(V)$ to ${\mathcal G}_{k'}(V')$
is an isometric embedding of $\Gamma_{k}(V)$ in $\Gamma_{k'}(V')$.
\end{theorem}

Some corollaries of Theorem \ref{theorem-main} will be given in Section 6.

\section{Intersections of $J(n,k)$-subsets}
\setcounter{equation}{0}

\subsection{Special subsets}
Let $X=\{x_{1},\dots,x_{n}\}$ be a $(2k)$-independent subset of a vector space $W$
(the dimension of $W$ is assumed to be not less than $2k$ and $n\ge 2k$) and let $k\ge 2$.
Consider the set ${\mathcal J}={\mathcal J}_{k}(X)$ formed by all $k$-dimensional subspaces spanned by subsets of $X$.
For every $i\in \{1,\dots,n\}$ we denote by ${\mathcal J}(+i)$ and ${\mathcal J}(-i)$
the sets consisting of all elements of ${\mathcal J}$ which contain $x_{i}$ and do not contain $x_{i}$, respectively.
Also, we write ${\mathcal J}(+i,+j)$ for the intersection of ${\mathcal J}(+i)$ and ${\mathcal J}(+j)$.
Every
$${\mathcal J}(+i,+j)\cup {\mathcal J}(-i),\;\;\;i\ne j$$
is said to be a {\it special} subset of ${\mathcal J}$.

We say that a subset ${\mathcal X}\subset {\mathcal J}$ is {\it inexact} if
there is a $(2k)$-independent $n$-element subset $Y\subset W$ such that
${\mathcal J}_{k}(Y)\ne {\mathcal J}$
(at least one of the vectors belonging to $Y$ is not a scalar multiple of a vector from $X$)
and ${\mathcal X}\subset {\mathcal J}_{k}(Y)$.

\begin{lemma}\label{lemma4-1}
Every inexact subset is contained in a special subset.
\end{lemma}

\begin{proof}
Let ${\mathcal X}$ be an inexact subset.
Denote by $S_i$ the intersection of all elements of ${\mathcal X}$ containing $x_{i}$
and set $S_{i}=0$ if there are no elements of ${\mathcal X}$ containing $x_{i}$.
There is at least one $i$ such that $S_{i}\ne \langle x_i \rangle$
(otherwise, ${\mathcal X}$ is not inexact).
Then $S_{i}=0$ or $\dim S_{i}\ge 2$.
In the first case, ${\mathcal X}$ is contained in  ${\mathcal J}(-i)$ which gives the claim.
If $\dim S_{i}\ge 2$ then  the inclusion
$${\mathcal X}\subset {\mathcal J}(+i,+j)\cup {\mathcal J}(-i)$$
holds for any $j\ne i$ such that $x_{j}\in S_{i}$.
\end{proof}

\begin{lemma}\label{lemma4-2}
If $X$ is independent then  the class of maximal inexact subsets coincides with
the class of special subsets.
\end{lemma}

\begin{proof}
By Lemma \ref{lemma4-1}, it sufficient to show that
every special subset is inexact.
Since $X$ is independent,
$$Y:=(X\setminus\{x_{i}\})\cup \{x_{i}+x_{j}\}$$
is independent and
${\mathcal J}_{k}(Y)$ contains
the special subset ${\mathcal J}(+i,+j)\cup {\mathcal J}(-i)$.
\end{proof}

\begin{rem}{\rm
Suppose that $R={\mathbb Z}_2$ and $X=\{x_{1},\dots,x_{5}\}$, where $x_{1},\dots,x_{4}$ are linearly independent vectors and
$$x_{5}=x_{1}+\dots+x_{4}.$$
Then $k=2$ and $X$ is $4$-independent.
The vectors
$x_{1}+x_{2},x_{3},x_{4},x_{5}$
are not linearly independent and $x_1$ can not be replaced by $x_{1}+x_{2}$ as in the proof of Lemma \ref{lemma4-2}.
The subspace $\langle x_1, x_2\rangle$ contains only three non-zero vectors --- $x_1,x_2,x_1+x_2$.
This means that ${\mathcal J}(+1,+2)\cup {\mathcal J}(-1)$
can not be inexact. The same arguments show that every special subset is not inexact.
}\end{rem}

\begin{rem}{\rm
It is not difficult to prove that all special subsets are inexact if $R$ is infinite, but we do not need
this fact.
}\end{rem}

The subsets ${\mathcal J}(+i,+j)$ and ${\mathcal J}(-i)$ are disjoint.
This means that every special subset contains precisely
$$a(n,k):=|{\mathcal J}(+i,+j)|+|{\mathcal J}(-i)|=\binom{n-2}{k-2}+\binom{n-1}{k}$$
elements. Lemma \ref{lemma4-1} implies the following.

\begin{lemma}\label{lemma4-3}
If an inexact subset consists of $a(n,k)$ elements
then it is a special subset.
\end{lemma}

A subset ${\mathcal X}\subset {\mathcal J}$ is said to be {\it complement}
if ${\mathcal J}\setminus {\mathcal X}$ is special, i.e.
$${\mathcal J}\setminus {\mathcal X}={\mathcal J}(+i,+j)\cup {\mathcal J}(-i)$$
for some distinct $i,j$. Then
$${\mathcal X}={\mathcal J}(+i)\cap {\mathcal J}(-j).$$
This complement subset will be denoted by ${\mathcal J}(+i,-j)$.

\begin{lemma}\label{lemma4-4}
Let $P,Q\in {\mathcal J}$. Then $d(P,Q)=m$ if and only if there are precisely
$$(k-m)(n-k-m)$$
distinct complement subsets of ${\mathcal J}$ containing both $P$ and $Q$.
\end{lemma}

\begin{proof}
The equality $d(P,Q)=m$ implies that
$$\dim(P\cap Q)=k-m\;\mbox{ and }\;\dim(P+Q)=k+m.$$
The complement subset ${\mathcal J}(+i,-j)$ contains both $P$ and $Q$
if and only if
$$x_{i}\in P\cap Q\;\mbox{ and }\;x_{j}\not\in P+Q.$$
So, there are precisely $k-m$ possibilities for $i$ and precisely $n-k-m$ possibilities for $j$.
\end{proof}

\subsection{Connectedness of the apartment graph}
Suppose that $1<k\le n-k$. If $X$ is a base of $V$ then ${\mathcal J}_{k}(X)$ is an apartment of ${\mathcal G}_{k}(V)$
and, by Lemma \ref{lemma4-2}, the class of maximal inexact subsets coincides with the class of special subsets.
Two apartments of ${\mathcal G}_{k}(V)$ are said to be {\it adjacent} if their intersection is a maximal inexact subset.
Consider the graph ${\rm A}_{k}$ whose vertices are apartments of ${\mathcal G}_{k}(V)$ and whose edges are pairs of adjacent apartments.

\begin{prop}\label{prop4-1}
The graph ${\rm A}_{k}$ is connected.
\end{prop}

\begin{proof}
Let  $B$ and $B'$ be bases of $V$.
The associated apartments of ${\mathcal G}_{k}(V)$ will be denoted by  ${\mathcal A}$ and ${\mathcal A}'$, respectively.
Suppose that ${\mathcal A}\ne{\mathcal A}'$ and show that these apartments can be connected in ${\rm A}_{k}$.

First we consider the case when $|B\cap B'|=n-1$.
Let
$$B=\{x_{1},\dots, x_{n-1},x_{n}\}\;\mbox{ and }\;B'=\{x_{1},\dots, x_{n-1},x'_{n}\}.$$
Since ${\mathcal A}\ne{\mathcal A}'$, the vector $x'_{n}$ is a linear combination of $x_{n}$ and some others $x_{i_{1}},\dots,x_{i_{m}}$.
Clearly, we can suppose that
$$x'_{n}=ax_{n} +\sum^{m}_{i=1}a_{i}x_{i}\;\mbox{ with }\;m \le n-1.$$
We prove the statement induction on $m$.
If $m=1$ then
$${\mathcal A}\cap {\mathcal A}'={\mathcal J}(+n,+1)\cup {\mathcal J}(-n)$$
is a maximal inexact subset and ${\mathcal A},{\mathcal A}'$ are adjacent.
Let $m\ge 2$. Denote by ${\mathcal A}''$ the apartment of ${\mathcal G}_{k}(V)$
associated with the base $x_{1},\dots, x_{n-1},x''_{n}$, where
$$x''_{n}:=ax_{n} +\sum^{m-1}_{i=1}a_{i}x_{i}.$$
By inductive hypothesis, ${\mathcal A}$ and ${\mathcal A}''$ can be can be connected in ${\rm A}_{k}$.
The equality $$x'_{n}=x''_{n}+a_{m}x_{m}$$ guarantees that
${\mathcal A}''$ and ${\mathcal A}'$ are adjacent.
This implies the existence of a path connecting ${\mathcal A}$ with ${\mathcal A}'$.

Now consider the case when $|B\cap B'|=m<n-1$ (possible $m=0$).
Suppose that
$$B\setminus B'=\{x_{1},\dots,x_{n-m}\}\;\mbox{ and }\;x'\in B'\setminus B.$$
For every $i\in \{1,\dots, n-m\}$ we define
$$S_{i}:=\langle B\setminus\{x_i\}\rangle.$$
Since the intersection of all $S_i$ coincides with $\langle B\cap B'\rangle$
and $x'$ does not belong to $\langle B\cap B'\rangle$,
there is at least one $S_i$ which does not contain $x'$.
Then
$$B_{1}:=(B\setminus \{x_{i}\})\cup \{x'\}$$
is a base of $V$. Denote by ${\mathcal A}_{1}$ the associated apartment of ${\mathcal G}_{k}(V)$.
It is clear that
$$|B\cap B_{1}|=n-1\;\mbox{ and }\;|B_{1}\cap B'|=m+1.$$
The apartment ${\mathcal A}_{1}$ coincides with ${\mathcal A}$ (if $x'$ is a scalar multiple of $x_{i}$)
or ${\mathcal A}$ and ${\mathcal A}_{1}$ are connected in ${\rm A}_{k}$.
Step by step we construct a sequence of bases
$$B=B_{0},B_{1},\dots, B_{n-m}=B'$$
such that $|B_{i-1}\cap B_{i}|=n-1$ for every $i\in \{1,\dots,n-m\}$.
Let ${\mathcal A}_i$ be the apartment of ${\mathcal G}_{k}(V
)$ associated with $B_{i}$.
Then for every $i\in \{1,\dots,n-m\}$
we have ${\mathcal A}_{i-1}={\mathcal A}_{i}$ or ${\mathcal A}_{i-1}$ and ${\mathcal A}_{i}$ are connected in ${\rm A}_{k}$.
This means that ${\mathcal A}={\mathcal A}_{0}$ and ${\mathcal A}'={\mathcal A}_{n-m}$ are connected in ${\rm A}_{k}$.
\end{proof}

\subsection{Intersections of $J(n,k)$-subsets of different types}
In this subsection we suppose that $W$ is a $(2k)$-dimensional vector space and $k\ge 2$.
Let
$$X=\{x_{1},\dots,x_{n}\}\;\mbox{ and }\;Y=\{y^{*}_{1},\dots,y^{*}_{n}\},\;\;n>2k$$
be $(2k)$-independent subsets of $W$ and $W^{*}$, respectively.
Denote by $U_{i}$ the annihilator of $y^{*}_{i}$. This is a $(2k-1)$-dimensional subspace of $W$. Suppose that the following conditions hold:
\begin{enumerate}
\item[$\bullet$] every $U_{i}$ is spanned by a subset of $X$,
\item[$\bullet$] every $\langle x_i\rangle$ is the intersection of some $U_{j}$.
\end{enumerate}
Since $X$ is a $(2k)$-independent subset, every $U_{i}$ is spanned by a $(2k-1)$-element subset $X_{i}\subset X$
and it does not contain any vector of $X\setminus X_i$.
Similarly, $Y$ is $(2k)$-independent and
every $x_i$ is contained in precisely $2k-1$ distinct $U_{j}$ whose intersection coincides with  $\langle x_i\rangle$.

We will investigate the intersection
$${\mathcal Z}:={\mathcal J}_{k}(X)\cap{\mathcal J}^{*}_{k}(Y).$$
It is formed by all elements of ${\mathcal G}_{k}(W)$ which are spanned by subsets of $X$
and can be presented as the intersections of $k$ distinct $U_j$.

We define
$$b(n,k):=\frac{\binom{2k-1}{k}n}{k}.$$
Note that this integer is not necessarily natural.
\begin{lemma}\label{lemma4-5}
$|{\mathcal Z}| \le b(n,k)$.
\end{lemma}

\begin{proof}
Denote by ${\mathcal Z}_{i}$ the set of all elements of ${\mathcal Z}$ containing $x_{i}$.
There are precisely $2k-1$ distinct $U_j$ containing $x_{i}$ and every element of ${\mathcal Z}$ is the intersection of $k$ distinct $U_{j}$.
This means that ${\mathcal Z}_{i}$ contains not greater than $\binom{2k-1}{k}$ elements.
Since every element of ${\mathcal Z}$ belongs to $k$ distinct ${\mathcal Z}_i$,
we have
$$|{\mathcal Z}|=\frac{|{\mathcal Z}_1|+\dots+|{\mathcal Z}_n|}{k}$$
which implies the required inequality.
\end{proof}

\begin{lemma}\label{lemma4-6}
$a(n,k)> b(n,k)$ except the case when $n=5$ and $k=2$.
\end{lemma}
\begin{proof}
We have
$$a(n,2)=1+\frac{(n-1)(n-2)}{2}=\frac{n^2-3n+4}{2}\;\mbox{ and }\; b(n,2)=\frac{3n}{2}.$$
An easy verification shows that the equality $a(n,2)> b(n,2)$ does not hold only for $n=5$.

From this moment we suppose that $k\ge 3$.
Then
$$a(n,k)=\binom{n-2}{k-2}+\binom{n-1}{k}=\frac{(n-2)!}{(k-2)!(n-k)!}+\frac{(n-1)!}{k!(n-k-1)!}=$$
$$=\frac{(n-2)\dots(n-k+1)\cdot k(k-1)}{k!}+\frac{(n-1)\dots(n-k)}{k!}=$$
$$=[k(k-1)+(n-1)(n-k)]\frac{(n-2)\dots(n-k+1)}{k!}$$
and
$$b(n,k)=\frac{\binom{2k-1}{k}n}{k}=\frac{n(2k-1)!}{k!k!}=\frac{n(2k-1)\dots(k+1)}{k!}=$$
$$=[n(k+1)]\frac{(2k-1)\dots(k+2)}{k!}.$$
Since $n\ge 2k+1$ and $k\ge 3$,
$$(n-1)(n-k)+k(k-1)=(n-1)(n-k)+(k+1)(k-1)-(k-1)\ge$$
$$\ge(n-1)(k+1)+(k+1)(k-1)-(k-1)=(n+k-2)(k+1)-(k-1)\ge$$
$$\ge (n+1)(k+1)-(k-1)=n(k+1)+2>n(k+1).$$
So,
\begin{equation}\label{eq4-1}
k(k-1)+(n-1)(n-k)>n(k+1).
\end{equation}
Also, $n\ge 2k+1$ implies that
$$n-2\ge 2k-1,\dots,n-k+1\ge k+2$$
and we have
\begin{equation}\label{eq4-2}
(n-2)\dots(n-k+1)\ge(2k-1)\dots(k+2).
\end{equation}
The inequality
$$a(n,k)=[k(k-1)+(n-1)(n-k)]\frac{(n-2)\dots(n-k+1)}{k!}>$$
$$>[n(k+1)]\frac{(2k-1)\dots(k+2)}{k!}=b(n,k)$$
follows from \eqref{eq4-1} and \eqref{eq4-2}.
\end{proof}

\begin{lemma}\label{lemma4-7}
If $n=5$ and $k=2$ then
$|{\mathcal Z}|\le 5<7=a(5,2)$.
\end{lemma}

\begin{proof}
In the present case, $U_1,\dots,U_{5}$ are $3$-dimensional, each $x_{i}$ is contained in precisely $3$ distinct $U_{j}$
and every element of ${\mathcal Z}$ is the intersection of $2$ distinct $U_j$.
If every $U_i$ contains not greater than $2$ elements of ${\mathcal Z}$
then $|{\mathcal Z}|\le \frac{2\cdot5}{2}=5$ (since every element of ${\mathcal Z}$ is contained in $2$ distinct $U_j$).

Suppose that $U_1$ is spanned by $x_1,x_2,x_3$ and contains $3$ elements of ${\mathcal Z}$.
These are $\langle x_1,x_2\rangle$, $\langle x_1,x_3\rangle$, $\langle x_2,x_3\rangle$.
Suppose that these subspaces are the intersections of $U_1$ with $U_2,U_3,U_4$.
Then each $x_i$, $i\in \{1,2,3\}$ is contained in $3$ distinct $U_{j}$, $j\in \{1,2,3,4\}$.
The subspace $U_5$ contains at least one of $x_i$, $i\in \{1,2,3\}$
and this $x_i$ is contained in $4$ distinct $U_j$, a contradiction.

The same arguments show that every $U_i$ contains not greater than $2$ elements of ${\mathcal Z}$ and we get the claim.
\end{proof}

Joining all results of this subsection, we get the following.

\begin{lemma}\label{lemma4-8}
$|{\mathcal Z}|<a(n,k)$.
\end{lemma}

\section{Proof of Theorem \ref{theorem-main}}

\setcounter{equation}{0}

Let $f$ be a $J$-mapping of ${\mathcal G}_{k}(V)$ to ${\mathcal G}_{k'}(V')$.

\begin{lemma}\label{lemma5-0}
The mapping $f$ is injective.
\end{lemma}

\begin{proof}
Let $P,Q$ be distinct elements of ${\mathcal G}_{k}(V)$.
We take an apartment ${\mathcal A}\subset {\mathcal G}_{k}(V)$ containing $P$ and $Q$.
Since $f({\mathcal A})$ is a $J(n,k)$-subset,
${\mathcal A}$ and $f({\mathcal A})$ have the same number of elements which implies that
$f(P)\ne f(Q)$.
\end{proof}

Consider the mapping $f_{*}$ which transfers every $P\in {\mathcal G}_{n-k}(V^{*})$ to $f(P^{0})$.
This is a $J$-mapping of ${\mathcal G}_{n-k}(V^{*})$ to ${\mathcal G}_{k'}(V')$.
It is clear that $f$ is an isometric embedding of $\Gamma_{k}(V)$ in $\Gamma_{k'}(V')$
if and only if
$f_{*}$ is an isometric embedding of $\Gamma_{n-k}(V^{*})$ in $\Gamma_{k'}(V')$.
Therefore, it sufficient to prove Theorem \ref{theorem-main} only in the case when $k\le n-k$.

Suppose that $k=1$, i.e. $f$ is a $J$-mapping of ${\mathcal G}_{1}(V)$ to ${\mathcal G}_{k'}(V')$.
Any distinct $P,Q\in {\mathcal G}_{1}(V)$ are adjacent and there is
an apartment ${\mathcal A}\subset {\mathcal G}_{1}(V)$ containing $P,Q$.
Since $f({\mathcal A})$ is a $J(n,1)$-subset,
$f(P)$ and $f(Q)$ are adjacent vertices of $\Gamma_{k'}(V')$.
Thus $f$ is an isometric embedding of $\Gamma_{1}(V)$ in $\Gamma_{k'}(V')$.

From this moment we suppose that $2\le k\le n-k$. By Subsection 2.4, we have
$$k\le \min\{k',n-k,n'-k'\}.$$

\begin{lemma}\label{lemma5-1}
If $n=2k$ then there exists $S\in {\mathcal G}_{k'-k}(V')$ such that the image of $f$ is contained in $[S\rangle_{k'}$.
\end{lemma}

\begin{proof}
Let ${\mathcal A}$ and ${\mathcal A}'$ be distinct apartments of ${\mathcal G}_{k}(V)$.
Then $f({\mathcal A})$ and $f({\mathcal A}')$ are $J(n,k)$-subsets and, since $n=2k$, Theorem \ref{theorem2-1} implies that
$$f({\mathcal A})=\Phi_{S}({\mathcal J}_{k}(X))\;\mbox{ and }\;f({\mathcal A}')=\Phi_{S'}({\mathcal J}_{k}(X')),$$
where $S,S'\in {\mathcal G}_{k'-k}(V')$ and $X,X'$ are independent $(2k)$-element subsets of $V'/S$ and $V'/S'$,
respectively. We need to show that $S=S'$.

By Proposition \ref{prop4-1}, it is sufficient to consider the case when
${\mathcal A}$ and ${\mathcal A}'$ are adjacent.
Then
$$|f({\mathcal A})\cap f({\mathcal A}')|=|{\mathcal A}\cap {\mathcal A}'|=a(2k,k)$$
and
$${\mathcal X}:=(\Phi_{S})^{-1}(f({\mathcal A})\cap f({\mathcal A}'))$$
is a subset of ${\mathcal J}_{k}(X)$ consisting of $a(2k,k)$ elements.
Since $S+S'$ is contained in all elements of $f({\mathcal A})\cap f({\mathcal A}')$,
every element of ${\mathcal X}$ contains $T:=(S+S')/S$.
If $S\ne S'$ then $t=\dim T\ge 1$ and
$$|{\mathcal X}|\le \binom{2k-t}{k-t}$$
which implies that
$$|{\mathcal X}|\le \binom{2k-1}{k-1}=\frac{(2k-1)!}{(k-1)!k!}=\binom{2k-1}{k}<\binom{2k-1}{k}+\binom{2k-2}{k-2}=a(2k,k),$$
a contradiction. Thus $S=S'$.
\end{proof}

\begin{lemma}\label{lemma5-2}
Suppose that $k<n-k$.
If $f$ transfers an apartment ${\mathcal A}\subset{\mathcal G}_{k}(V)$ to a $J(n,k)$-subset of first type
then
the images of all apartments of ${\mathcal G}_{k}(V)$ are $J(n,k)$-subsets of first type and
there exists $S\in {\mathcal G}_{k'-k}(V')$ such that the image of $f$ is contained in $[S\rangle_{k'}$.
\end{lemma}

\begin{proof}
By our hypothesis,
$$f({\mathcal A})=\Phi_{S}({\mathcal J}_{k}(X)),$$
where $S\in {\mathcal G}_{k'-k}(V')$ and $X$ is
a $(2k)$-independent subset of $V'/S$ consisting of $n$ vectors
$$\overline{x}_1=x_{1}+S,\dots,\overline{x}_n=x_{n}+S.$$
Denote by $S_i$ the $(k'-k+1)$-dimensional subspace of $V'$ corresponding to $\overline{x}_i$.
Every element of $f({\mathcal A})$ is the sum of $k$ distinct $S_j$.

Let ${\mathcal A}'$ be an apartment of ${\mathcal G}_{k}(V)$ distinct from ${\mathcal A}$.
We need to show that $f({\mathcal A}')$ is a $J(n,k)$-subset of first type and is contained in $[S\rangle_{k'}$.
By Proposition \ref{prop4-1}, it is sufficient to consider the case when
${\mathcal A}$ and ${\mathcal A}'$ are adjacent.
As in the proof of the previous lemma,
$${\mathcal X}:=(\Phi_{S})^{-1}(f({\mathcal A})\cap f({\mathcal A}'))$$
is a subset of ${\mathcal J}_{k}(X)$
consisting of $a(n,k)$ elements.
There are the following possibilities:
\begin{enumerate}
\item[(1)] ${\mathcal X}$ is contained in a special subset of ${\mathcal J}_{k}(X)$,
\item[(2)] there is no special subset of ${\mathcal J}_{k}(X)$ containing ${\mathcal X}$.
\end{enumerate}

{\it Case} (1).
Every special subset of ${\mathcal J}_{k}(X)$ consists of $a(n,k)=|{\mathcal X}|$ elements.
This implies that ${\mathcal X}$ is a special subset of ${\mathcal J}_{k}(X)$.
Suppose that
$${\mathcal X}={\mathcal J}(+i,+j)\cup {\mathcal J}(-i)$$
(see Subsection 4.1 for the notation).
We take any $(k-1)$-dimensional subspace $T\subset V'/S$ spanned by a subset of $X$ containing $\overline{x}_j$.
Then
$${\mathcal S}:={\mathcal J}_{k}(X)\cap [T\rangle_{k}$$
is a star of ${\mathcal J}_{k}(X)$ contained in ${\mathcal X}$
(if $P\in {\mathcal S}$ contains $\overline{x}_{i}$ then it belongs to ${\mathcal J}(+i,+j)$
and $P\in {\mathcal S}$ is an element of ${\mathcal J}(-i)$ if it does not contain $\overline{x}_{i}$).

Consider $\Phi_{S}({\mathcal S})$. This is a star of $f({\mathcal A})$.
By Lemma \ref{lemma2-2}, this star consists of $n-k+1$ vertices (since $f({\mathcal A})$ is a $J(n,k)$-subset of first type).
Also, it is contained in
$\Phi_{S}({\mathcal X})\subset f({\mathcal A}')$ and
Lemma \ref{lemma2-2} guarantees that $f({\mathcal A}')$ is a $J(n,k)$-subset of first type.

We take $P,Q\in {\mathcal X}$ such that $P\cap Q=0$.
The intersection of $\Phi_{S}(P)$ and $\Phi_{S}(Q)$ coincides with $S$.
Since $\Phi_{S}(P)$ and $\Phi_{S}(Q)$ both belong to $f({\mathcal A}')$ and $f({\mathcal A}')$ is a $J(n,k)$-subset of first type,
the associated $(k'-k)$-dimensional subspace of $V'$ coincides with $S$ and
$f({\mathcal A}')$ is contained in $[S\rangle_{k'}$.

{\it Case} (2).
For every $i\in \{1,\dots,n\}$ the intersection of all elements of ${\mathcal X}$
containing $\overline{x}_{i}$ coincides with $\langle \overline{x}_{i}\rangle$
(otherwise, as in the proof of Lemma \ref{lemma4-1} we show that ${\mathcal X}$ is contained in a special subset of ${\mathcal J}_{k}(X)$ which is impossible).
Then the intersection of all elements of
$$\Phi_{S}({\mathcal X})=f({\mathcal A})\cap f({\mathcal A}')$$
containing $S_i$ coincides with $S_{i}$.
This implies that the intersection of all elements of $f({\mathcal A})\cap f({\mathcal A}')$ is $S$.

Therefore, if $f({\mathcal A}')$ is a $J(n,k)$-subset of first type
then the associated $(k'-k)$-dimensional subspace of $V'$ coincides with $S$, i.e.
$f({\mathcal A}')$ is contained in $[S\rangle_{k'}$.
Then ${\mathcal X}$ is an inexact subset of ${\mathcal J}_{k}(X)$.
By Lemma \ref{lemma4-3}, ${\mathcal X}$ is a special subset of ${\mathcal J}_{k}(X)$ which is impossible.

So, $f({\mathcal A}')$ is a $J(n,k)$-subset of second type. Then
$$f({\mathcal A}')=\Phi^{U}({\mathcal J}^{*}_{k}(Y)),$$
where $U\in {\mathcal G}_{k'+k}(V')$ and $Y$ is a $(2k)$-independent subset of $U^{*}$
consisting of $n$ vectors $y^{*}_{1},\dots,y^{*}_{n}$.
Denote by $U_{i}$ the annihilator of $y^{*}_{i}$ (in $U$).
By Remark \ref{rem2-1}, every element of $f({\mathcal A}')$ is the intersection of $k$ distinct $U_j$.

The set
\begin{equation}\label{eq5-1}
(\Phi^{U})^{-1}(f({\mathcal A})\cap f({\mathcal A}'))
\end{equation}
is contained in ${\mathcal J}^{*}_{k}(Y)$.
Denote by ${\mathcal Y}$ the subset of  ${\mathcal J}_{k}(Y)$ formed by the annihilators of all elements of \eqref{eq5-1}.
It consists  of $a(n,k)$ elements.
If ${\mathcal Y}$ is contained in a special subset of ${\mathcal J}_{k}(Y)$ then
it coincides with this special subset. In this case, there is a star ${\mathcal S}\subset {\mathcal J}_{k}(Y)$ contained in ${\mathcal Y}$.
Let ${\mathcal S}^{0}$ be the subset of ${\mathcal J}^{*}_{k}(Y)$ consisting of the annihilators of all
elements of ${\mathcal S}$.
Then $\Phi^{U}({\mathcal S}^{0})$ is a top of $f({\mathcal A}')$ contained in $f({\mathcal A})\cap f({\mathcal A}')$.
This contradicts Lemma \ref{lemma2-2}, since $f({\mathcal A})$ and $f({\mathcal A}')$ are $J(n,k)$-subsets of different types.

Thus there is no special subset of ${\mathcal J}_{k}(Y)$ containing ${\mathcal Y}$.
This means that for every $i\in\{1,\dots,n\}$ the intersection of all elements of ${\mathcal Y}$
containing $y^{*}_i$ coincides with $\langle y^{*}_i\rangle$.
By Lemma \ref{lemma2-1}, $U_i$ is the sum of the annihilators (in $U$) of these elements; hence it is
the sum of some elements of $f({\mathcal A})\cap f({\mathcal A}')$.
Since every element of $f({\mathcal A})$ is the sum of $k$ distinct $S_j$,
\begin{enumerate}
\item[(*)] every $U_i$ is the sum of some $S_{j}$.
\end{enumerate}
This implies that every $U_{i}$ contains $S$ (since $S$ is contained in all $S_i$) and $f({\mathcal A}')$ is a subset of $[S\rangle_{k'}$
(every element of $f({\mathcal A}')$ is the intersection of $k$ distinct $U_j$).

Since the intersection of all elements of
$f({\mathcal A})\cap f({\mathcal A}')$ containing $S_i$ coincides with $S_{i}$
and every element of $f({\mathcal A}')$ is the intersection of $k$ distinct $U_j$,
\begin{enumerate}
\item[(**)] every $S_{i}$ is the intersection of some $U_{j}$.
\end{enumerate}
Then every $S_i$ is contained in $U$ and
$f({\mathcal A})$ is a subset of $\langle U]_{k'}$ (since every element of $f({\mathcal A})$ is the sum of $k$ distinct $S_j$).

So, $f({\mathcal A})$ and $f({\mathcal A}')$ both are contained in $[S,U]_{k'}$.
The vector space $W:=U/S$ is $2k$-dimensional.
It is clear that
$$f({\mathcal A})=\Phi^{U}_{S}({\mathcal J}_{k}(X))\;\mbox{ and }\;f({\mathcal A}')=\Phi^{U}_{S}({\mathcal J}^{*}_{k}(Y')),$$
where $Y'$ is the $(2k)$-independent $n$-element subset of $W^{*}$ induced by $Y$.
The annihilators of the vectors belonging to $Y'$ are $U_{i}/S$, $i\in \{1,\dots,n\}$.
The facts (*) and (**) guarantee that $X$ and $Y'$ satisfy the conditions of Subsection 4.3:
\begin{enumerate}
\item[$\bullet$] the annihilator of every element of $Y'$ is spanned by a subset of $X$,
\item[$\bullet$] every $\langle \overline{x}_{i}\rangle$ is the intersection of the annihilators of some elements from $Y'$.
\end{enumerate}
By Subsection 4.3,
$${\mathcal Z}:={\mathcal J}_{k}(X)\cap {\mathcal J}^{*}_{k}(Y')$$
contains less than $a(n,k)$ elements. This contradicts the fact that
$$\Phi^{U}_{S}({\mathcal Z})=f({\mathcal A})\cap f({\mathcal A}')$$
consists of  $a(n,k)$ elements.
So, the case (2) is impossible.
\end{proof}

By Lemma \ref{lemma5-2}, if $k<n-k$ then
the images of all apartments of ${\mathcal G}_{k}(V)$ are $J(n,k)$-subsets of the same type.

Suppose that one of the following possibilities is realized:
\begin{enumerate}
\item[$\bullet$] $n=2k$,
\item[$\bullet$] $k<n-k$ and the images of all apartments of ${\mathcal G}_{k}(V)$ are $J(n,k)$-subsets of first type.
\end{enumerate}
By Lemmas \ref{lemma5-1} and \ref{lemma5-2}, the image of $f$ is contained in $[S\rangle_{k'}$ with $S\in {\mathcal G}_{k'-k}(V')$.
This implies the existence of a mapping
$$g:{\mathcal G}_{k}(V)\to {\mathcal G}_{k}(V'/S)$$
such that $f=\Phi_{S}\circ g$.
This is a $J$-mapping which transfers every apartment of ${\mathcal G}_{k}(V)$
to a certain $J_{k}(X)$, where $X$ is a $(2k)$-independent $n$-element subset of $V'/S$.
Using results of Subsection 4.1, we prove the following.

\begin{lemma}\label{lemma5-3}
The mapping $g$ is an isometric embedding of $\Gamma_{k}(V)$ in $\Gamma_{k}(V'/S)$.
\end{lemma}

\begin{proof}
Let $P,Q\in {\mathcal G}_{k}(V)$ and let ${\mathcal A}$ be an apartment of ${\mathcal G}_{k}(V)$
containing $P$ and $Q$.
If ${\mathcal X}$ is a special subset of ${\mathcal A}$ then
${\mathcal X}={\mathcal A}\cap {\mathcal A}'$, where ${\mathcal A}'$ is an apartment of ${\mathcal G}_{k}(V)$
adjacent with ${\mathcal A}$.
By Subsection 4.1,
$$g({\mathcal X})=g({\mathcal A})\cap g({\mathcal A}')$$
is an inexact subset of $g({\mathcal A})$.
It consists of $a(n,k)$ elements  and Lemma \ref{lemma4-3} implies that
$g({\mathcal X})$ is a special subset of $g({\mathcal A})$.
Since ${\mathcal A}$ and $g({\mathcal A})$ have the same number of special subsets,
a subset of ${\mathcal A}$ is special if and only if its image is a special subset of $g({\mathcal A})$.
Then ${\mathcal X}$ is a complement subset of ${\mathcal A}$ if and only if $g({\mathcal X})$ is a complement subset of $g({\mathcal A})$.
Lemma \ref{lemma4-4} implies that
$$d(P,Q)=d(g(P),g(Q))$$
and we get the claim.
\end{proof}

Since $\Phi_{S}$ is an isometric embedding of $\Gamma_{k}(V'/S)$ in $\Gamma_{k'}(V')$,
Lemma \ref{lemma5-3} guarantees that $f=\Phi_{S}\circ g$ is an isometric embedding of $\Gamma_{k}(V)$ in $\Gamma_{k'}(V')$.

Now suppose that $k<n-k$ and
the images of all apartments of ${\mathcal G}_{k}(V)$ are $J(n,k)$-subsets of second type.
Consider the mapping $f^{*}$ which sends every $P\in {\mathcal G}_{k}(V)$ to $f(P)^{0}$.
This is a $J$-mapping of ${\mathcal G}_{k}(V)$ to  ${\mathcal G}_{n'-k'}(V'^{*})$.
It transfers every apartment of ${\mathcal G}_{k}(V)$ to a $J(n,k)$-subset of first type.
Then $f^{*}$ is  an isometric embedding of $\Gamma_{k}(V)$ in $\Gamma_{n'-k'}(V'^{*})$
which means that $f$ is an isometric embedding of $\Gamma_{k}(V)$ in $\Gamma_{k'}(V')$.

\section{Strong $J$-mappings}

A $J$-mapping of ${\mathcal G}_{k}(V)$ to ${\mathcal G}_{k'}(V')$ is said to be {\it strong}
if there is an apartment of ${\mathcal G}_{k}(V)$ whose image  is an apartment in a parabolic subspace of ${\mathcal G}_{k'}(V')$.
The apartments preserving mappings considered in \cite[Section 3.4]{Pankov-book} are strong $J$-mappings.

If $n=2k\ge 4$ then every $J$-mapping of ${\mathcal G}_{k}(V)$ to ${\mathcal G}_{k'}(V')$
is strong (Theorem \ref{theorem2-1}) and, by Theorems \ref{theorem2-2} and \ref{theorem-main},
it is induced by a semilinear embedding of $V$ in $U/S$
or a semilinear embedding of $V$ in $(U/S)^{*}$, where
$$S\in {\mathcal G}_{k'-k}(V')\;\mbox{ and }\;U\in {\mathcal G}_{k'+k}(V').$$
In this section, we show that all strong $J$-mappings of ${\mathcal G}_{k}(V)$ to ${\mathcal G}_{k'}(V')$
are induced by semilinear embeddings if $1<k<n-1$.
For $k=1,n-1$ this fails \cite{HK}.

First we prove the following generalization of \cite[Theorem 3.10]{Pankov-book}.

\begin{cor}\label{cor1}
If $n=n'$ and $1<k<n-1$ then every strong $J$-mapping of ${\mathcal G}_{k}(V)$ to ${\mathcal G}_{k}(V')$
is induced by a semilinear embedding of $V$ in $V'$ or a semilinear embedding of $V$ in $V'^{*}$
and the second possibility can be realized only in the case when $n=2k$.
\end{cor}

\begin{proof}
Let $f$ be a strong $J$-mapping of ${\mathcal G}_{k}(V)$ to ${\mathcal G}_{k}(V')$.
By Theorem  \ref{theorem-main}, $f$ is an isometric embedding of $\Gamma_{k}(V)$ in $\Gamma_{k}(V')$.
We suppose that $n=n'$ and $1<k<n-1$.
Then there is an apartment ${\mathcal A}\subset {\mathcal G}_{k}(V)$
such that $f({\mathcal A})$ is an apartment of ${\mathcal G}_{k}(V')$.

In the case when $n=2k$, the required statement follows from Theorem \ref{theorem2-2}.

If $k<n-k$ then, by Theorem \ref{theorem2-2}, we have the following possibilities:
\begin{enumerate}
\item[$\bullet$] $f=(l)_{k}$, where $l:V\to V'$ is a semilinear $(2k)$-embedding;
\item[$\bullet$] $f=(s)^{*}_{k}$, where $s:V\to U^{*}$ is a semilinear $(2k)$-embedding
and $U$ is a $(2k)$-dimensional subspace of $V'$.
\end{enumerate}
In the second case, the image of $f$ is contained in $\langle U]_{k}$.
Since $2k<n$, $\langle U]_{k}$ does not contain any apartment of ${\mathcal G}_{k}(V')$.
So, this case is impossible and $f=(l)_{k}$.
Then $l$ transfers any base of $V$ associated with ${\mathcal A}$ to a base of $V'$.
This implies that $l$ is a semilinear embedding.

Let $k>n-k$. Consider the  mapping which transfers every $P\in {\mathcal G}_{n-k}(V^{*})$ to  $f(P^{0})^{0}$.
This is a strong $J$-mapping of ${\mathcal G}_{n-k}(V^{*})$ to ${\mathcal G}_{n-k}(V'^{*})$.
By the arguments given above, it is induced by a semilinear embedding $s:V^{*}\to V'^{*}$.
Denote by $g$ the mapping  of the set of all subspaces of $V$ to the set of all subspaces of $V'$
which sends every $P$ to $s(P^{0})^{0}$.
By \cite[Subsection 3.4.3]{Pankov-book}, it is induced by a semilinear embedding $l:V\to V'$,
i.e.
$$g(P)=\langle l(P)\rangle$$
for every subspace $P\subset V$.
Since the restriction of $g$ to ${\mathcal G}_{k}(V)$ coincides with $f$,
we have $f=(l)_{k}$.
\end{proof}

\begin{cor}\label{cor2}
Suppose that $1<k<n-1$ and $n\ne 2k$. Then for every strong $J$-mapping $f:{\mathcal G}_{k}(V)\to {\mathcal G}_{k'}(V')$
one of the following possibilities is realized:
\begin{enumerate}
\item[{\rm (1)}] there exist $S\in {\mathcal G}_{k'-k}(V')$ and $U\in {\mathcal G}_{n+k'-k}(V')$
such that $f=\Phi^{U}_{S}\circ (l)_{k}$, where $l:V\to U/S$ is a semilinear embedding;
\item[{\rm (2)}] there exist $S'\in {\mathcal G}_{n'-k'-k}(V'^{*})$ and $U'\in {\mathcal G}_{n+n'-k'-k}(V'^{*})$ such that
$f={\rm A}\circ\Phi^{U'}_{S'}\circ (l)_{k}$, where $l:V\to U'/S'$ is a semilinear embedding and ${\rm A}$
is the annihilator mapping of ${\mathcal G}_{n'-k'}(V'^{*})$ to ${\mathcal G}_{k'}(V')$.
\end{enumerate}
\end{cor}

\begin{proof}
By Theorem \ref{theorem-main}, $f$ is an isometric embedding of $\Gamma_{k}(V)$ in $\Gamma_{k'}(V')$.
Suppose that $k<n-k$. Theorem \ref{theorem2-2} states that one of the following possibilities is realized:
\begin{enumerate}
\item[$\bullet$] $f=\Phi_{S}\circ (l)_{k}$,
where $S\in {\mathcal G}_{k'-k}(V')$ and $l:V\to V'/S$ is semilinear $(2k)$-embedding;
\item[$\bullet$] $f=\Phi^{U}\circ (s)^{*}_{k}$, where $U\in {\mathcal G}_{k'+k}(V')$ and $s:V\to U^{*}$ is a semilinear $(2k)$-embed\-ding.
\end{enumerate}
As in Corollary \ref{cor1}, we establish that $l$ and $s$ both are semilinear embeddings.

We get a mapping of type (1) in the first case.

In the second case, 
the image of $f$ is contained in $[T,U]_{k'}$, where $T\in {\mathcal G}_{k+k'-n}(V')$ is the annihilator of $s(V)$ in $U$.
Consider the mapping $f^{*}$  sending every $P\in {\mathcal G}_{k}(V)$ to $f(P)^{0}$.
The image of this mapping is contained in $[S',U']_{n'-k'}$ with
$$S':=U^{0}\in {\mathcal G}_{n'-k'-k}(V'^{*})\;\mbox{ and }\;U':=T^{0}\in {\mathcal G}_{n+n'-k'-k}(V'^{*}).$$
Then $f^{*}=\Phi^{U'}_{S'}\circ g$, where $g$ is a $J$-mapping of ${\mathcal G}_{k}(V)$ to ${\mathcal G}_{k}(U'/S')$.
This $J$-mapping is strong (since $f$ and $f^{*}$ are strong $J$-mappings).
The dimension of $U'/S'$ is equal to $n$ and Corollary \ref{cor1} implies that $g$ is induced by a semilinear embedding of $V$ in $U'/S'$.
Thus $f$ is a mapping of type (2).

Now suppose that $k>n-k$. The image of $f$ coincides with the image of the mapping $f_{*}$ which transfers every $P\in {\mathcal G}_{n-k}(V^{*})$ to $f(P^{0})$.
This image is contained in
$$[N,M]_{k'},\;\;\;N\in{\mathcal G}_{k'-n+k}(V'),\;\;M\in{\mathcal G}_{k'+k}(V')$$
($f_{*}$ is a mapping of type (1)) or it is a subset of
$$[S,U]_{k'},\;\;\;S\in {\mathcal G}_{k'-k}(V'),\;\;U\in {\mathcal G}_{k'+n-k}(V')$$
($f_{*}$ is a mapping of type (2)).

In the second case, $f=\Phi^{U}_{S}\circ g$, where $g$ is a strong $J$-mapping of ${\mathcal G}_{k}(V)$ to ${\mathcal G}_{k}(U/S)$.
Since $U/S$ is $n$-dimensional,
Corollary \ref{cor1} implies that $g$ is induced by a semilinear embedding of $V$ in $U/S$ and $f$ is a mapping of type (1).

Suppose that the image of $f$ is contained in $[N,M]_{k'}$.
As above, we consider the mapping $f^{*}$ which sends every $P\in {\mathcal G}_{k}(V)$ to $f(P)^{0}$.
Its image is a subset of $[S',U']_{n'-k'}$ with
$$S':=M^{0}\in {\mathcal G}_{n'-k'-k}(V'^{*})\;\mbox{ and }\;U':=N^{0}\in {\mathcal G}_{n+n'-k'-k}(V'^{*}).$$
Then $f^{*}=\Phi^{U'}_{S'}\circ g$, where $g$ is a strong $J$-mapping of ${\mathcal G}_{k}(V)$ to ${\mathcal G}_{k}(U'/S')$.
The standard arguments show that $f$ is a mapping of type (2).
\end{proof}

\end{document}